\documentclass[reqno]{amsart}
\calclayout
 
\usepackage{mathtools,amssymb,amsfonts,amsmath,latexsym,tikz}
  \usepackage[alphabetic]{amsrefs}
\usepackage{tikz-cd}
\usepackage{hyperref}
\usepackage{enumerate}
\usepackage{mathtools}
\usepackage{physics}
\usepackage{xcolor}

\usepackage[bb=libus]{mathalpha}

\theoremstyle{definition}
\newtheorem{theorem}{Theorem}[section]

\newtheorem{proposition}[theorem]{Proposition}
\newtheorem{lemma}[theorem]{Lemma}
\newtheorem{corollary}[theorem]{Corollary}

\newtheorem{remark}[theorem]{Remark}

\newcommand{\w}{\omega}
\newcommand{\tw}{\tilde{\omega}}
\newcommand{\W}{\Omega}

\newcommand{\p}{\partial}
\newcommand{\bp}{\bar{\partial}}
\newcommand{\pd}{{\partial^{\dagger}}}
\newcommand{\bpd}{{\bar{\partial}^{\dagger}}}

\newcommand{\C}{\mathbb{C}}

\newcommand{\R}{\mathbb{R}}

\newcommand{\D}{\mathcal{D}}
\newcommand{\F}{\mathcal{F}}
\renewcommand{\O}{\mathcal{O}}

\newcommand{\X}{\mathcal{X}}
\newcommand{\Y}{\mathcal{Y}}
\newcommand{\Z}{\mathcal{Z}}

\newcommand{\bbi}{\mathbb{i}}

\renewcommand{\Tr}{\mathrm{Tr}}

\newcommand{\ks}{\mathrm{KS}}

\newcommand{\id}{\mathrm{id}}
\newcommand{\End}{\mathrm{End}}

\newcommand{\im}{\mathrm{Im}}

\allowdisplaybreaks

\begin{document}

\title{The Aeppli Parameter for the Heterotic Moduli}

\author{Pei-Lin Wu}
\date{\today}
\thanks{Department of Mathematics, UBC, 1984 Mathematics Road,
    Vancouver, BC, Canada, \href{peilinwu@math.ubc.ca}{peilinwu@math.ubc.ca}\\Comments are welcome!}

\begin{abstract}
  In this paper, we construct a family of solutions to the $3$-fold Hull-Strominger system using the Aeppli class, without introducing the auxiliary gauge connection on the tangent bundle. In particular, we deform the conformally balanced metric along an Aeppli class off-shell and then tune it by a hermitian $(1,1)$-form dependent on this Aeppli class on-shell to satisfy the anomaly cancellation condition. The existence of the family of solutions is then obtained by the implicit function theorem. This refines the previous work \cite{PW24} by not introducing auxiliary gauge connection, thereby matching the expected dimension with the Bott-Chern parameter. This construction of the Aeppli parameter also extends to the $n$-fold Hull-Strominger system.
\end{abstract}

\maketitle
% \tableofcontents
\section{Introduction and Outline}
Let $X$ be a compact complex 3-fold, with a nowhere vanishing holomorphic $(3,0)$-form $\W$ and a hermitian metric $\w$. Let $E\to X$ be a holomorphic vector bundle with the hermitian metric $h$, the associated Chern connection $A_{h}$ and the Chern curvature $F_{h}$. It is physically motivated to consider the following system of partial differential equations (PDEs):
\begin{align}
    d(||\W||_{\w}\w^2)&=0\quad\text{Conformally balanced condition}\label{eqn:cbc},\\
    F_{h}\wedge \w^2=0,\ F^{0,2}_{h}&=0\quad\text{Hermitian-Yang-Mills conditions}\label{eqn:hym},\\
    \bbi\p\bp\w-\alpha'(\Tr R_{\w}^2 - \Tr F_{h}^{2})&=0\quad \text{Anomaly cancellation condition}\label{eqn:acc}.
\end{align}
This set of PDEs arises from the consideration of the heterotic string theory \cite{CHSW85}, in particular, with the introduction of the flux by \cites{H86,S86}. It is called the Hull-Strominger system or the heterotic system in the literature and generalizes the Calabi-Yau (CY) and Hermitian-Yang-Mills (HYM) equations to a coupled nonlinear system. The parameter space of solutions modulo diffeomorphisms is called the heterotic moduli. 
\begin{remark}
    One has to be slightly careful for the choice of connections. Physically, it is more natural to use the Hull connection, while mathematically easier to use the Chern connection. Since \cite{MP25} has concluded the equivalence of the Chern connection and the Hull connection for the Hull-Strominger system up to order ${\alpha'}^2$, we will use the Chern connection for our heterotic system.
\end{remark}

When $\alpha'=0$, the anomaly cancellation condition reduces to the pluriclosed condition $\p\bp\w=0$. Together with the conformally balanced condition, this leads to $d\omega=0$ by a non-trivial calculation \cite{IP13}. Therefore, we are in the territory of K\"ahler geometry. Given a compact K\"ahler Calabi-Yau 3-fold $X$, since $\w$ is closed, one can associate a K\"ahler class $[\w]\in H^{1,1}_{\mathrm{dR}}(X,\R)$. Then Yau's theorem \cite{Y78} gives the existence and uniqueness of a K\"ahler Ricci-flat metric in the same K\"ahler class $\w_{\mathrm{CY}}\in [\w]$. Then by Donaldson-Uhlenbeck-Yau (DUY) theorem \cites{D85,UY86}, there exists a unique connection $A_{\mathrm{HYM}}$ satisfying HYM conditions on the stable holomorphic bundle $E\to X$ with respect to $\w_{\mathrm{CY}}$. These results combined gave the first existence result of the Hull-Strominger system in the K\"ahler case $\alpha'=0$. 

When $\alpha'\neq 0$, there is non-zero flux, i.e. non-zero torsion, and we venture into the non-K\"ahler realm. We are interested in obtaining the analogue of Yau's theorem and DUY theorem in the non-K\"ahler case. Many works have been done on the existence of the Hull-Strominger system on non-K\"ahler manifolds. Solutions via the large radius expansion were studied in \cite{WW87}. The first existence result of exact solution came from \cite{LY05} via the local perturbation and the implicit function theorem. \cite{FY08} gave an explicit construction on torus fibrations over K3 surfaces, and \cite{FHP17} on torus fibrations over Riemannian surfaces. 

In the perspective of the moduli space, \cites{AGS14,dlOS14a} introduced algebroid structures and described the moduli space via holomorphic structures on an extension bundle. \cites{CdlOM17,CdlOM19} derived the explicit metric with a K\"ahler potential on the heterotic moduli space with the presupposed smooth and complex structure. The physical arguments suggest under suitable conditions, the moduli admits a K\"ahler structure. There is also a recent result \cite{MY26} along this line to compute the heterotic moduli metric up to order ${\alpha'}^{2}$. It is interesting that despite the metric receives corrections from the deformation of the Hull connection mixing the complex structure and the hermitian moduli, the K\"ahler potential obtained in \cite{CdlOM17} remains unchanged.

As for the deformation theory, \cite{AdlOMS-CS18} considered finite deformations of the Hull-Strominger system and reached an $L_{3}$ structure on the moduli. \cites{G-FRST22,G-FRT17} investigated the infinitesimal moduli and give the first construction of the Aeppli parameter for the deformation of the balanced metric satisfying the Hull-Strominger system. Later, the moment map picture of the gauge theory of the extended algebroid was done in \cites{G-FRT20,G-FRT24}. Among those works, the construction of the Aeppli parameter requires extension bundles which introduced the auxiliary gauge connection in the tangent bundle, effectively changing the term $\Tr R_{\w}^2$.

Our main result is a partial analog of Yau's theorem and DUY theorem for the Hull-Strominger system in the non-K\"ahler case. In particular, our result removes the dependency on the auxiliary gauge connection in the tangent bundle.
\begin{theorem}
    Let $E\to X$ be a stable holomorphic vector bundle over a compact complex 3-fold admitting a K\"ahler form $\w$ and a nowhere vanishing holomorphic $(3,0)$-form $\W$, and satisfying the topological constraints $c_{1}(TX)=c_{1}(E)=0$ and $c_2(TX)=c_{2}(E)$. For sufficiently small $\alpha'$, there exists a solution $(\tw,\tilde{h})$ to the Hull-Strominger system whose $\alpha'$-corrected Aeppli class is the same as the Aeppli class of the reference K\"ahler metric $\w$,
    \begin{align}
        [\tw,\tilde{h}]_{\mathrm{A},\alpha'}=[\w]_{\mathrm{A}}\in H_{\mathrm{A}}^{1,1}(X,\R).
    \end{align}
\end{theorem}
In the above theorem, the $\alpha'$-corrected Aeppli class is defined in \eqref{eqn:alpha'-corrected Aeppli class}.
\begin{remark}
    The above theorem is not a full analog of Yau's theorem as the question of the uniqueness of the solution is still open. In \cite{G-FRT24}, the condition for the uniqueness of the solution is identified yet very hard to apply. One is left wondering whether the uniqueness can be obtained by some complicated non-linear analysis, or there are indeed multiple solutions in the same Aeppli class away from the K\"ahler point.
\end{remark}
We call a pair of reference K\"ahler metric and HYM connection $(\w,A_{h})$ solving the Hull-Strominger system a K\"ahler point on the heterotic moduli. By Chern correspondence, we can identify $A_{h}$ via $h$ or vice versa, therefore we will also denote the K\"ahler point by $(\w,h)$. We will deform it by an Aeppli class to obtain a deformation path of the metric $\w$ and prove its existence via the implicit function theorem near a K\"ahler point.
\begin{theorem}
    Near a K\"ahler point $(\w,h)$ on the heterotic moduli with fixed complex structure, for sufficiently small $\alpha'$, there exists a path of solutions to the Hull-Strominger system parameterized by the Aeppli class $\mathfrak{a}$
    \begin{align}
        [\tw,\tilde{h}]_{\mathrm{A},\alpha'}&=[\w]_{\mathrm{A}}+\mathfrak{a}\in H_{\mathrm{A}}^{1,1}(X,\R).
    \end{align}
\end{theorem}
The above theorem can be interpreted as a local description of the heterotic moduli near a K\"ahler point where the Aeppli class $\mathfrak{a}$ is a local coordinate. Moreover, the tangent space of the partial heterotic moduli with a fixed complex structure and a fixed holomorphic structure $\bar{\D}_0$ at a K\"ahler point is isomorphic to the Aeppli cohomology $H^{1,1}_{\mathrm{A}}(X,\R)$. In particular, since no auxiliary gauge connection of the tangent bundle is introduced, this agrees with the expected dimension \cite{PW24} of the partial heterotic moduli given by the Bott-Chern cohomology $H^{2,2}_{\mathrm{BC}}(X,\R)$ via the Poincar\'e-type duality between the Aeppli and Bott-Chern cohomologies. 

If we further assume the unobstructedness of the deformation of the stable holomorphic bundle $E\to X$, then we can also deform the holomorphic structure $\bar{\D}_{0}$ of the bundle smoothly to obtain a family of holomorphic structure $\bar{\D}_{\alpha}$. Then the above theorem extends to a joint deformation of the metric and bundle connection parameterized by the Aeppli class and holomorphic class respectively. 
\begin{theorem}
    Near a K\"ahler point $(\w,h)$ on the heterotic moduli with fixed complex structure, for sufficiently small $\alpha'$, there exists a family of solutions to the Hull-Strominger system parameterized by the Aeppli class $\mathfrak{a}\in H^{1,1}_{\mathrm{A}}(X,\R)$ and the holomorphic class $\alpha\in H^{1}(X,\End E)$.
\end{theorem}
This can be interpreted as a local description of the heterotic moduli near a K\"ahler point where the Aeppli class $\mathfrak{a}$ and the holomorphic class $\alpha$ are local coordinates. Moreover, the tangent space of the partial heterotic moduli with fixed complex structure near a K\"ahler point is isomorphic to $H^{1,1}_{\mathrm{A}}(X,\R)\oplus H^{1}(X,\End E)$, which agrees with the expected dimension given by the Bott-Chern cohomology $H^{2,2}_{\mathrm{BC}}(X,\R)\oplus H^{1}(X,\End E)$ via the Poincar\'e-type duality between the Aeppli and Bott-Chern cohomologies.

\textbf{Outlines:} This paper is structured as follows. In section \ref{sec:1} we quickly review the two natural cohomology classes related to the Hull-Strominger system, with emphasis on the Aeppli cohomology and the construction of the $\alpha'$-corrected Aeppli class. In section \ref{sec:2} we motivate and give the joint deformation ansatz with the Aeppli parameter $\mathfrak{a}\in H^{1,1}_{\mathrm{A}}(X)$. In section \ref{sec:3} we set up the implicit function theorem with appropriate spaces and calculate the linearization and conclude the main theorem, with a brief discussion on the bundle moduli with the unobstructedness assumption where the bundle parameter $\alpha\in H^{1}(X,\End E)$ enters in the joint deformation. In section \ref{sec:4} for a streamlined comparison, we give a brief discussion of the Bott-Chern parameter, and the duality with the Aeppli parameter. We end the paper with some immediate open questions and future directions in section \ref{sec:5}. In the appendix \ref{App:B}, we give a brief discussion of the canonical construction of the Bott-Chern secondary characteristics and the smoothness of the implicit function. In the appendix \ref{App:A}, we give some useful lemmas of the reflection formula in $n$-fold for $1$-forms and $2$-forms. In the appendix \ref{App:C}, we will give a brief discussion of the extension of the main theorem to the $n$-fold Hull-Strominger system.

\textbf{\textit{Acknowledgement:}} The author benefits from many discussions with S\'ebastien Picard in the process of formulating and finishing this paper. The author also thanks Mario Garcia-Fernandez for comments on the extension bundles during a conference at Tsinghua, and Jock McOrist for discussions about the physical significance of the Hull connection and interesting results on the protected K\"ahler potential.

\section{Aeppli Class}\label{sec:1}
\subsection{Two natural cohomology classes}
By different partitions of the Hull-Strominger system, we arrive at two natural choices of the cohomology classes to parametrize the heterotic moduli, namely:
\begin{enumerate}
    \item If we would like to make an ansatz that satisfies the conformally balanced condition\footnote{If one include the joint deformation of the bundle, then one needs to include the partial HYM condition $\tilde{F}_{\tilde{h}}^{0,2}=0$ as well.},
    \begin{align}
        d(||\W||_{\tw}\tw^2)=0,
    \end{align}
    we naturally arrive at the Bott-Chern class 
    \begin{align}
        [||\W||_{\tw}\tw^2]_{\mathrm{BC}}\in H_{\mathrm{BC}}^{2,2}(X,\R)=\frac{\ker(d)\cap \W^{2,2}_{\R}X}{\Im(\p\bp)}.
    \end{align}
   Then the implicit function theorem can be applied to study the vanishing locus of the remaining equations. This is the approach taken in \cites{CPY22,PW24}.
    \item If we would like to make an ansatz that satisfies the anomaly cancellation condition\footnote{Technically speaking, one needs to include the partial HYM condition $\tilde{F}_{\tilde{h}}^{0,2}=0$ to make the form type consistent in the ansatz. But we suppress this to focus on the parallel comparison here.},
    \begin{align}
        \bbi\p\bp\tw-\alpha'(\Tr \tilde{R}_{\tw}^2 - \Tr \tilde{F}_{\tilde{h}}^{2})=0,
    \end{align}
    when $\alpha'=0$, we naturally arrive at the Aeppli class
    \begin{align}
        [\w]_{\mathrm{A}}\in H_{\mathrm{A}}^{1,1}(X,\R)=\frac{\ker(\p\bp)\cap \W^{1,1}_{\R}X}{\Im(\p\oplus\bp)}.
    \end{align}
    But when $\alpha'\neq0$, one needs a careful construction of the Aeppli class, since away from the K\"ahler point $\tw\not\in\ker(\p\bp)$. In general, in the following subsection, we will use the Bott-Chern secondary characteristics to construct the $\alpha'$-corrected Aeppli class. Then the implicit function theorem can be applied to study the vanishing locus of the remaining equations. This is the approach taken in \cites{G-FRST22,PW24} but with the introduction of the auxiliary gauge connection in the tangent bundle, and thus modifying the $\Tr R_{\tw}^2$ term.
\end{enumerate}

\subsection{$\alpha'$-Corrected Aeppli class}
We will work with the Aeppli cohomology defined by
\begin{align}
    H^{1,1}_{A}(X)&=\frac{\ker(\p\bp)\cap \W^{1,1}X}{\mathrm{Im}(\p\oplus\bp)}.
\end{align}
The Aeppli class of any pluriclosed $(1,1)$-form $\tw$ is denoted by $[\tw]_{\mathrm{A}}$. In particular, we can consider the Aeppli class of the K\"ahler metric $\w$, denoted by $[\w]_{\mathrm{A}}$. Since $\w$ is K\"ahler, we have $\p\bp\w=0$, and thus $[\w]_{\mathrm{A}}$ is well-defined. 

However, for a general solution $(\tw,\tilde{h})$ to the Hull-Strominger system, $\tw$ is not pluriclosed by the anomaly cancellation condition, and we need a careful construction to define its Aeppli class. This leads to the following definition of the $\alpha'$-corrected Aeppli class $[\tw,\tilde{h}]_{\mathrm{A},\alpha'}$,
\begin{align}\label{eqn:alpha'-corrected Aeppli class}
    [\tw,\Tilde{h}]_{\mathrm{A},\alpha'}&:=\qty[\tw-\alpha'C_{2}(\tw,\w)+\alpha' C_{2}[\Tilde{h},h]+\alpha'\beta]_{\mathrm{A}}\in H^{1,1}_{\mathrm{A}}(X),
\end{align}
where $C_{2}$'s are the Bott-Chern secondary characteristics defined by
\begin{align}
    \bbi\p\bp C_{2}[\tw,\w]&=\Tr \Tilde{R}_{\tw}^{2}-\Tr R_{\w}^2,\\
    \bbi\p\bp C_{2}[\Tilde{h},h]&=\Tr \Tilde{F}_{\Tilde{h}}^{2}-\Tr F_{h}^2,
\end{align}
and $\beta\in \W^{1,1}(X)$ such that
\begin{align}\label{eqn:beta}
    \bbi\p\bp\beta= \Tr F_{h}^{2}-\Tr R_{\w}^{2}.
\end{align}
We may also take the real part if necessary since $\bbi\p\bp\Im(\beta)=0$. The canonical construction of the Bott-Chern secondary characteristics is briefly discussed in the Appendix \ref{App:B}. By construction,
\begin{align}\label{aeppli}
    \bbi\p\bp[\tw,\Tilde{h}]_{\mathrm{A},\alpha'}=0.
\end{align}
In other words, if $(\tw,\tilde{h})$ is a solution to the Hull-Strominger system, then the $\alpha'$-corrected Aeppli class $[\tw,\tilde{h}]_{\mathrm{A},\alpha'}$ is well-defined with respect to the reference $(\w,h)$. Moreover, one can show that it is independent of the choice of the reference $(\w,h)$ \cite{P24}. 

\subsection{Duality} As aforementioned, there is a duality between the Aeppli cohomology $H^{1,1}_{\mathrm{A}}$ and Bott-Chern cohomology $H^{2,2}_{\mathrm{BC}}$ via the pairing analogue to the Poincar\'e duality:
\begin{align}\label{eqn:Poincare-Duality}
    H^{1,1}_{\mathrm{A}}(X,\C)\times H^{2,2}_{\mathrm{BC}}(X,\C)&\to \C, \nonumber\\
    ([\tw]_{\mathrm{A}},[||\W||_{\tw}\tw^2]_{\mathrm{BC}})&\mapsto\int_{X}[\tw]_{\mathrm{A}}\wedge[||\W||_{\tw}\tw^2]_{\mathrm{BC}}.
\end{align}
Due to this, it is natural to give a deformation ansatz of the balanced metric using the Bott-Chern parameter, which was investigated in \cite{PW24}. The duality suggests two parameterization of Aeppli and Bott-Chern classes should be of the same dimension. This requires one to give Aeppli parametrization without introducing the auxiliary gauge connection. 

This motivates us to investigate a refined description of the Aeppli deformation without introducing the auxiliary gauge connection. Once we establish it, we will give a brief parallel comparison between the Aeppli class and the Bott-Chern class approaches to the heterotic moduli, and use this duality to conclude the dimension matching of the local heterotic moduli near a K\"ahler point.

\section{Deformation Ansatz}\label{sec:2}
\subsection{Aeppli deformation}
We consider the following Aeppli deformation ansatz for the metric $\tw$ and the hermitian metric $\tilde{h}$ such that
\begin{align}
    [\tw]_{\mathrm{A},\alpha'} = [\w]_{\mathrm{A}}+\mathfrak{a},
\end{align}
or equivalently,
\begin{align}
    \tw = \w+\mathfrak{a}+\theta-\alpha'(C_{2}[\tilde{h},h]-C_2[\tw,\w]+\beta),
\end{align}
where $\theta\in\Im(\p\oplus\bp)\cap \W^{1,1}_{\R}(X)$. 

This has been investigated in the previous work \cite{PW24}. However, our previous approach and many other works \cites{G-FRST22,G-FRT24} requires the introduction of an auxiliary gauge connection $\gamma$ on $T^{1,0}X$ and the modified anomaly cancellation relation
\begin{align}
    \bbi \p\bp\tw = \alpha'(\Tr \tilde{R}_{\gamma}^2 - \Tr \tilde{F}_{\tilde{h}}^2),
\end{align}
for suitable setup of the implicit function theorem. This effectively enlarges the heterotic moduli.

It was noticed in \cite{P24} that the above modified ansatz indeed admits solutions to $\tilde{\w}$ via a two-step implicit function theorem argument, indicating the possibility of refining the deformation to absorb the dependence of $\gamma$ and hence removing the dependency of this auxiliary gauge connection. 

We here provide the conjectured refinement in a one-step implicit function theorem argument. The main idea is to decouple and allow a joint $\W^{1,1}_{\R}X$-deformation $u$ and Aeppli-deformation $\mathfrak{a}$ off-shell. But here the $\W^{1,1}_{\R}$-deformation $u(\mathfrak{a})$ will turn out to be solely dependent on the Aeppli deformation (and on the bundle deformation) once the solution exists. This dependence is determined on-shell while preserving the anomaly cancellation condition, so no additional degrees of freedom are introduced.

The deformation path is given by the following: take $u\in\W^{1,1}_{\R}X$, deform the metric by
\begin{align}
    \tw_{u}=\w+u,
\end{align}
This hermitian $(1,1)$-form $\tilde{\w}_{u}$ is not solving the anomaly cancellation condition in general. But if we require its $\alpha'$-corrected Aeppli class to be in the same Aeppli class of $[\w]_{\mathrm{A}}$ on-shell up to a difference $\mathfrak{a}\in H^{1,1}_{\mathrm{A}}(X)$, then it will solve the anomaly cancellation equation. In other words, we require the following on-shell condition for the deformation path:
\begin{align}
    [\tilde{\w}_{u}]_{\mathrm{A},\alpha'}=[\w]_{\mathrm{A}}+\mathfrak{a},
\end{align}
or equivalently,
\begin{align}
    \w+u=\w+\mathfrak{a}+\theta-\alpha'(C_{2}[\tilde{h},h]-C_{2}[\tw_{u},\w]+\beta).
\end{align}
Therefore, $u$ is not an independent parameter on-shell, but depends on the Aeppli parameter $\mathfrak{a}$ as the follows, 
\begin{align}\label{u(a)}
    u(\mathfrak{a})=\mathfrak{a}+\theta-\alpha'(C_{2}[\tilde{h},h]-C_{2}[\tw_{u},\w]+\beta).
\end{align}
It is worth noting that the above equation \eqref{u(a)} is a non-linear PDE for $u$, and the solution $u(\mathfrak{a})$ is solely dependent on the Aeppli parameter $\mathfrak{a}$ on-shell. One can use the implicit function theorem to show the existence and uniqueness of $u(\mathfrak{a})$ for sufficiently small $\mathfrak{a}$ \cite{P24}. Therefore, the deformation path is effectively parameterized by the Aeppli parameter $\mathfrak{a}$. The Aeppli parameter $\mathfrak{a}$ provides a moduli coordinate locally around a K\"ahler point. 

For later setup of the implicit function theorem, it is convenient to decouple the above equation \eqref{u(a)} by defining:
\begin{align}
    \tw_{u}&=\w+u,\\
    \chi&=\w+\mathfrak{a}+\theta-\alpha'(C_{2}[\tilde{h},h]-C_{2}[\tw_{u},\w]+\beta),
\end{align}
and $\theta\in \im(\p\oplus\bp)\cap \W^{1,1}_{\R}(X)$. Then the on-shell condition is equivalent to $\tw_{u}-\chi=0$, which means exactly the anomaly cancellation condition is satisfied. After the application of the implicit function theorem, then we can therefore describe the vanishing locus, i.e. partial heterotic moduli, with the Aeppli parameter $\mathfrak{a}$ as a local coordinate. Effectively, we no longer need to introduce the auxiliary gauge connection $\gamma$ on $T^{1,0}X$ and extend the bundle $E$. 

We now take $\theta=\p\eta^{0,1}+\bp\gamma^{1,0}$ for some $\eta^{0,1}$ and $\gamma^{1,0}$. The hermitian condition $\theta=\bar{\theta}$ requires that $\overline{\eta^{0,1}}=\gamma^{1,0}$. Due to this hermitian condition, as noticed in \cite{G-FRST22}, one can write $\theta$ in a more compact form by some $1$-form $\xi$ as
\begin{align}
    \theta=(1+J)d\xi=2(d\xi)^{1,1},\quad \xi\in\Im(d^{\dagger})\cap\W^{1}_{\R}X,
\end{align}
where $J$, by abusing the notation, is the induced operator on $k$-forms by almost complex structure $J$,
\begin{align}
    J=\sum_{p+q=k}\bbi^{p-q}\Pi^{p,q},
\end{align}
and here $\Pi^{p,q}:\W^{k}X\to \W^{p,q}X$ is the type projection operator from $k$-forms to $(p,q)$-forms.

\subsection{Bundle deformation}
We now specify the deformation of the bundle connection from the parameter space $H^{1}(X,\End E)$ in a way that the deformed curvature $\tilde{F}_{\tilde{h}}$ automatically satisfies part of the Hermitian-Yang-Mills condition
\begin{align}
    \tilde{F}_{\tilde{h}}^{0,2}=0.
\end{align}

Given $E\to X$ a vector bundle with reference hermitian metric $h$, let $\bar{\D}$ be a holomorphic structure on $E$. By the Chern correspondence, the operator $\bar{\D}$ uniquely determines a connection $D=\D+\bar{\D}$ which is metric compatible with respect to $h$. 

We deform the metric by
\begin{align}
    \tilde{h}=he^{v},\quad v\in\Gamma(\End E),\ v^{\dagger}=v,
\end{align}
where $\dagger$ denotes the adjoint with respect to $h$. For later convenience, we will denote $v\in \Gamma_{\R}(\End E)$. 

The associated curvature then reads
\begin{align}
    \tilde{F}_{\tilde{h}} = \bp(\tilde{h}^{-1}\p \tilde{h}),\quad \tilde{h}=he^{v}.
\end{align}

\section{Implicit Function $\F$ and its validity}\label{sec:3}
We will first investigate the deformation of balanced metric via the Aeppli class with a fixed complex structure and a fixed holomorphic class. Then with the unobstructedness assumption, the construction can be extended to the joint deformation of the balanced metric and bundle connection.

For our purpose, we need the implicit function theorem that works for infinite dimensional Banach space. We will apply the implicit function theorem (IFT) appeared in \cite{MS12}.

\subsection{Setup}
We aim to study the neighborhood of a K\"ahler point on the slice of the heterotic moduli with a fixed complex structure and a fixed bundle connection. First, we fix background data. Let $\W$ be a holomorphic volume form. Let $\w$ be a K\"ahler Ricci-flat metric and let $h$ be a HYM metric on $E$ with respect to $\w$. Then, we consider the following spaces:
\begin{align}
    \F &:\X\times \Y\to \Z,\nonumber\\
    \X &= \R \times H^{1,1}_{\mathrm{A}}(X),\\
    \Y &= \W^{1,1}_{\R}X \times (\Im(d^{\dagger})\cap\W^{1}_{\R}X) \times \Gamma_{\R}(\End_0 E),\\
    \Z &= \W^{1,1}_{\R}X \times (\Im(d^{\dagger})\cap\W^{1}_{\R}X) \times V(E),
\end{align}
where $\W^{1,1}_{\R}X$ is the space of hermitian (1,1)-form, $\Gamma_{\R}(\End_0 E)=\{u\in \Gamma_{\R}(\End E):\Tr (u)=0,u^{\dagger}=u\}$ and $V(E)$ is the trace-free, self-adjoint $\mathrm{End}E$-valued-6-form,
\begin{align}
    V(E):=\{v\in\W^{6}(\End E):\int_{X}\Tr(v)=0, v^{\dagger}=v\}.
\end{align}
\begin{remark}
    Here the $\End_{0}E$ is needed since the ansatz is not sensitive to shifts $v\mapsto v+C\id_{E}$ by constant multiples of identity.
\end{remark}
With the spaces, we denote the moduli parameters $\mathrm{X}=(\alpha',\mathfrak{a})\in\X$ and the cohomology equivalence class coordinates $\mathrm{Y}=(u,\xi,v)\in\Y$. 

We now consider the function
\begin{align}
    \F(\mathrm{X},\mathrm{Y})=\mqty[
    \tw_{u}-\chi\\
    \star_{\w}d(||\W||_{\chi}\chi^2)\\
    ||\W||_{\chi}e^{\frac{v}{2}}\qty(\chi^2\wedge i\Tilde{F}_{v})e^{-\frac{v}{2}}-C\chi^3\otimes \mathrm{id}
    ]=\mqty[\F_1\\ \F_2\\ \F_3] \in \Z,
\end{align}
with the normalization constant
\begin{align}
    C=\frac{1}{\mathrm{Rk}(E)}\frac{\int_{X}||\W||_{\chi}\chi^2\wedge \Tr (i\Tilde{F}_{\Tilde{h}})}{\int_{X}||\W||_{\chi}\chi^3},
\end{align}
where 
\begin{align}
    \tw_{u}&=\w+u,\\
    \chi&=\w+\mathfrak{a}+\theta(\xi)-\alpha'(C_{2}[\tilde{h},h]-C_{2}[\tw_{u},\w]+\beta),\quad \theta(\xi)=(1+J)d\xi,\\
    \tilde{h}&=he^{v}.
\end{align}
Note that $\F(\mathrm{X},\mathrm{Y})=0$ is exactly the constraint equation that the deformed $(\tw,\tilde{h})$ solves the Strominger system. Hence locally around a K\"ahler point, the heterotic moduli on a slice of a fixed complex structure, is given by $\F^{-1}(0)$. In particular, $\F_{1}=0$ exactly enforces the anomaly cancellation condition, $\F_{2}=0$ enforces the conformally balanced condition and $\F_{3}=0$ enforces the remaining HYM condition.

In our setup above, to apply the implicit function theorem, we need to verify the surjectivity of $D\F$,
\begin{align}
    D\F=\mqty[ D_{\mathrm{X}}\F & D_{\mathrm{Y}}\F ]
\end{align}
By a rank consideration, the surjectivity of $D\F$ is given by the invertibility of the block operator $D_{\mathrm{Y}}\F$.

\subsection{Well-posedness}
We now confirm the well-definedness of the source and target spaces in above setup of the implicit function theorem. By this, we mean to construct the suitable Banach spaces viable for the implicit function theorem. To make them into suitable Banach space, we need to equip them with appropriate H\"older norms and therefore consider the following spaces:

\begin{align}
    \X&=\R\times H^{1,1}_{\mathrm{A}}(X),\\
    \Y'&=C^{k,\gamma}(\W^{1,1}_{\R}X)\times C^{k+2,\gamma}(\Im(d^{\dagger})\cap \W^{1}_{\R}X)\times C^{k+2,\gamma}(\Gamma_{\R}(\End_0 E))\\
    \Z'&=C^{k,\gamma}(\W^{1,1}_{\R}X)\times C^{k,\gamma}(\Im(d^{\dagger})\cap \W^{1}_{\R}X)\times C^{k,\gamma}(V(E))
\end{align}
$C^{k,\gamma}(\W^{p,q}_{\R}X)$ is a well-known Banach space. Also, it is standard via the Hodge decomposition and real projection that $\Im(d^{\dagger})\cap \W^{1}_{\R}X$ is Banach.

We also give a brief account about the smoothness of $\F$ in Appendix \ref{App:B}. Hence, the $\F$ map is well-posed for the implicit function theorem.

\subsection{Linearization of $\F$}
Before we compute the linearization of $\F$, we first check $\F(\mathrm{X}=0,\mathrm{Y}=0)=0$ at the K\"ahler point. We have
\begin{align}
    \F_{1}&=\w-\w=0,\\
    \F_{2}&=\star_{\w}d(||\W||_{\w}\w^2)=0,\\
    \F_{3}&=||\W||_{\w}(\w^2\wedge \bbi F^{1,1})=0.
\end{align}
We can now compute the linearization of $\F$ which will be denoted by $D_{\mathrm{Y}}\F$, where
\begin{align}
    \left. D_{\mathrm{Y}}\F \right|_{(\mathrm{X},\mathrm{Y})=(0,0)}\mqty[\dot{u} \\ \dot{\xi} \\ \dot{v}] = \mqty[
        \id_{\W^{1,1}_{\R}X} & L_{4} & 0\\
        0 & L_1 & 0\\
        0 & L_3 & L_2
    ],
\end{align}
and (suppressing the H\"older norm)
\begin{align}
    L_{1}&: \Im(d^{\dagger}) \cap \W^{1}_{\R}X \to \Im(d^{\dagger}) \cap \W^{1}_{\R}X, \\
    L_{2}&: \Gamma_{\R}(\End_0 E)\to V(E), \\
    L_{3}&:\Im(d^{\dagger}) \cap \W^{1}_{\R}X \to V(E),\\
    L_{4}&:\W^{1,1}_{\R}X \to \Im(d^{\dagger}) \cap \W^{1}_{\R}X.
\end{align}
To apply the implicit function theorem, it suffices to show that $D_{\mathrm{Y}}\F$ is invertible. Since $\F$ is smooth, its linearization is bounded. The inverse of the block matrix $D_{\mathrm{Y}}\F$ is given by
\begin{align}
    D_{\mathrm{Y}}\F|_{(\mathrm{X},\mathrm{Y})=(0,0)}^{-1}=\mqty[
        \id & -L_{4}L_{1}^{-1} & 0\\
        0 & L_1^{-1} & 0\\
        0 & -L_2^{-1}L_3L_1^{-1} & L_2^{-1}
    ]
\end{align}
Hence, the invertibility of $D_{\mathrm{Y}}\F$ is equivalent to that of $L_{1}, L_{2}$. We will calculate them respectively below.

\subsection{Calculation of $L_{2}$ and its invertibility}
We start with the linearization of the bundle connection part $L_2$. Since the explicit calculation was given in \cite{PW24}, we only record the statement here.
\begin{lemma}
    The linearization of the bundle connection part $L_2$ is given by
    \begin{align}
        L_{2}\dot{v}&=\left. \frac{\p \F_{3}}{\p v} \right|_{(\mathrm{X},\mathrm{Y})=(0,0)}=\delta_{0}\F_{3}=\delta_{0}\qty(||\W||_{\chi}e^{\frac{v}{2}}(\chi^2\wedge i\tilde{F}_{v})e^{-\frac{v}{2}}-C\chi^3\otimes \id)=\bbi \Lambda_{\w}(\bar{\D}\D\dot{v})\otimes ||\W||_{\w}\frac{\w^3}{3!},
    \end{align}
    where
    \begin{align}
        \D e^{v}=\p e^{v} + [ A , e^{v} ],\quad A = h^{-1}\p h.
    \end{align}
    Hence
    \begin{align}
        L_2\bullet&=\bbi\Lambda_{\w}(\bar{\D}\D\bullet)\otimes ||\W||_{\w}\frac{\w^3}{3!}.
    \end{align}
\end{lemma}
Its invertibility has also been established in \cite{PW24}.
\begin{lemma}
    The above $L_{2}:\Gamma_{\R}(\End_{0} E)\to V(E)$ is invertible.
\end{lemma}

\subsection{Calculation of $L_{1}$ and its invertibility}
\begin{lemma}
    The linearization of the balanced metric deformation $L_1$ is given by
    \begin{align}
        L_1\dot{\xi}&=\left. \frac{\p \F_2}{\p \xi}\right|_{(\mathrm{X},\mathrm{Y})=(0,0)}=\delta_{0}\F_{2}=\delta_{0}\qty( \star_{\w}d(||\W||_{\chi(\xi)}\chi(\xi)^2) )=2||\W||_{\w}d^{\dagger}d\dot{\xi}=2||\W||_{\w}\Delta\dot{\xi}.
    \end{align}
    Hence
    \begin{align}
        L_1\bullet&=2||\W||_{\w}\Delta\bullet.
    \end{align}
\end{lemma}

\begin{proof}
    We refer to the proof of Lemma. \ref{lemma-C1} in general $n$-fold case, and then take $n$=3.
\end{proof}

We see that the linearization of the $\star$-dual conformally balanced condition is just the usual Laplacian $\Delta = d^{\dagger}d + dd^{\dagger}$. Then by the standard elliptic theory argument, we can deduce that $L_1$ is bijective and hence we can apply the implicit function theorem. 
\begin{lemma}
    The above $L_1:\Im(d^{\dagger})\cap \W^{1}_{\R}X\to \Im(d^{\dagger})\cap \W^{1}_{\R}X$ is invertible. 
\end{lemma}
One quick way to see it is invertible is by the standard elliptic theory for the Laplacian operator and the Hodge decomposition. 
\begin{remark}
    In \cite{G-FRST22}, the variation of conformal balanced condition $d(||\W||_{\chi}\chi^2)$ was considered and computed, where $T$ has been defined and the linearization operator $dL_{\w}T(1+J)d$ has been shown to be elliptic and invertible by considering an elliptic complex and computing of its symbol. In comparison, the above $\star$-dual balanced condition gives a simpler operator and easier to show the ellipticity and the invertibility for the application of the implicit function theorem.
\end{remark}

\subsection{Main result}
Combining the invertibility of $L_1$ and $L_2$ with the implicit function theorem, we obtain the existence of a family of solutions the Hull-Strominger system near a K\"ahler point.

For $\mathfrak{a}=0$, we simply have the existence of a solution in the $\alpha'$-corrected Aeppli class.
\begin{theorem}
    Let $E\to X$ be a stable holomorphic vector bundle over a compact complex 3-fold admitting a K\"ahler form $\w$ and a nowhere vanishing holomorphic $3$-form $\W$, and satisfying the topological constraints $c_{1}(TX)=c_{1}(E)=0$ and $c_2(TX)=c_{2}(E)$. For sufficiently small $\alpha'$, there exists a solution $\tw$ to the Hull-Strominger system whose $\alpha'$-corrected Aeppli class is the same as the Aeppli class of the reference K\"ahler metric $\w$,
    \begin{align}
        [\tw,\tilde{h}]_{\mathrm{A},\alpha'}=[\w]_{\mathrm{A}}\in H_{\mathrm{A}}^{1,1}(X,\R).
    \end{align}
\end{theorem}
For $\mathfrak{a}\neq 0$, in perspective of the moduli space, we have constructed a family of solutions to the heterotic system near a K\"ahler point with the parameter space of dimension $h^{1,1}_{\mathrm{A}}(X)$, which is the dimension of the Aeppli cohomology group $H_{\mathrm{A}}^{1,1}(X,\R)$.
\begin{theorem}
    \label{thm:joint}
    Near a K\"ahler point $(\w,h)$ on the heterotic moduli with fixed complex structure, for sufficiently small $\alpha'$, there exists a path of solutions to the Hull-Strominger system parameterized by the Aeppli class $\mathfrak{a}$
    \begin{align}
        [\tw,\tilde{h}]_{\mathrm{A},\alpha'}=[\w]_{\mathrm{A}}+\mathfrak{a}\in H_{\mathrm{A}}^{1,1}(X,\R).
    \end{align}
\end{theorem}

\subsection{With unobstructedness assumption}
If the deformation of the bundle is unobstructed, we can use the similar argument to show the existence of the solutions of the heterotic system with a fixed complex structure. 

In particular, the unobstructedness assumption ensures that $H^{1}(X,\End E)$ parametrizes local solutions to 
$\tilde{F}^{0,2}_{\tilde{h}}=0$ nearby a K\"ahler point, and hence locally parameterized solutions to this partial HYM equation.

Given $E\to X$ a vector bundle with reference hermitian metric $h$, let $\bar{\D}_{\alpha}$ be such a smooth family of holomorphic structure on $E$ varying with a parameter $\alpha\in H^{1}(X,\End E)$. By the Chern correspondence, the operator $\bar{\D}_{\alpha}$ uniquely determines a connection $D_{\alpha}=\D_{\alpha}+\bar{\D}_{\alpha}$ which is metric compatible with respect to $h$. 

Recall that an adjoint of an endomorphism $\gamma\in \Gamma(\End E)$ by definition satisfies
\begin{align}
    \langle \gamma x,y \rangle_{h} = \langle x,\gamma^{\dagger} y \rangle_{h},
\end{align}
then we can express
\begin{align}
    \gamma^{\dagger} = (h\gamma h^{-1})^{*}.
\end{align}
Then, the inner product with respect to the deformed metric $\tilde{h}$ is related to the reference one by
\begin{align}
    \langle x , y\rangle_{\tilde{h}} = \langle e^{v}x , y \rangle_{h}.
\end{align}
By the Chern correspondence, it associates to a holomorphic structure $\bar{\D}_{\alpha}$, which is a unique $\tilde{h}$-metric compatible connection
\begin{align}
    D_{\alpha,v}=\D_{\alpha,v}+\bar{\D}_{\alpha},\quad \D_{\alpha,v}=\D_{\alpha}+e^{-v}\D_{\alpha}e^{v}.
\end{align}
In other words, the deformation of the holomorphic structure is described by a parameter $\alpha\in H^{1}(X,\End E)$.

We only need to consider the same function $\F$ but with deformation spaces
\begin{align}
    \X'&=\R\times H^{1,1}_{\mathrm{A}}(X)\times H^{1}(X,\End E),
\end{align}
but with this bundle deformation
\begin{align}
    \tilde{h}=he^{v},\quad \D_{\alpha,v}=\D_{\alpha}+e^{-v}\D_{\alpha}e^{v}.
\end{align}
Then the calculation of the linearization and the application of the implicit function theorem follows similarly as above and in \cite{PW24} to conclude the following main theorem.
\begin{theorem}
    Near a K\"ahler point $(\w,h)$ on the heterotic moduli with fixed complex structure, for sufficiently small $\alpha'$, there exists a path of solutions to the Hull-Strominger system parameterized by the Aeppli class $\mathfrak{a}\in H^{1,1}_{\mathrm{A}}(X,\R)$ and the holomorphic class $\alpha\in H^{1}(X,\End E)$.
    \begin{align}
        \tw&\in[\tw,\tilde{h}]_{\mathrm{A},\alpha'}=[\w]_{\mathrm{A}}+\mathfrak{a}\in H_{\mathrm{A}}^{1,1}(X,\R),\\
        \tilde{h}&=he^{v},\ \D_{\alpha,v}=\D_{\alpha}+e^{-v}\D_{\alpha}e^{v},
    \end{align}
    and the parameter space is of dimension $h^{1,1}_{\mathrm{A}}(X)+h^1(X,\End E)$.
\end{theorem}

\section{Parallel Comparison}\label{sec:4}
\subsection{Bott-Chern parameter}
We now can give a brief parallel treatment and comparison between the Bott-Chern parameter and Aeppli parameter. In \cite{PW24} the Bott-Chern parameter gives the following deformation path on the heterotic moduli with a fixed complex structure near a K\"ahler point: The Bott-Chern cohomology class $H^{2,2}_{\mathrm{BC}}(X)$, which is defined by
\begin{align}
    H^{2,2}_{\mathrm{BC}}(X)=\frac{\ker(d)\cap\W^{2,2}X}{\Im(\p\bp)\cap\W^{1,1}X}.
\end{align}
Then given a balanced metric $\tw$, we can produce a Bott-Chern class
\begin{align}
    H^{2,2}_{\mathrm{BC}}(X)\ni [||\W||_{\tw}\tw^2]_{\mathrm{BC}}=||\W||_{\w}\w^2+[\mathfrak{b}]=||\W||_{\w}\w^2+\mathfrak{b}+\Theta,
\end{align}
where
\begin{align}
    (\mathfrak{b}+\Theta)\in \W^{2,2}_{\R}X : d\mathfrak{b} = 0, \Theta=\bbi\p\bp\gamma\in\p\bp\W^{1,1}_{\R}X.
\end{align}
It was shown in \cite{M82} that for $(\mathfrak{b}+\Theta)$ small enough, taking the square root defines a positive hermitian metric $\tw$. Then we have the deformed metric satisfies the balanced conditions automatically,
\begin{align}
    d(||\W||_{\tw}\tw^2)=d\mathfrak{b}+d(\bbi\p\bp\gamma)=0.
\end{align}
By \cites{M82}, one can obtain the precise formula of $\w$ by taking the square root. Therefore, the implicit function theorem was set up to solve the conformally balanced condition together the Hermitian-Yang-Mills as following:
Hence the parameters $\mathrm{X}$ and $\mathrm{Y}$ read,
\begin{align}
    \mathrm{X}&=(\alpha',\mathfrak{b},\alpha)\in \mathbb{R}\times \mathbb{H}^{2,2} \times H^{1}({\rm End} \, E)=\X,\\
    \mathrm{Y}&=(\Theta,v)\in \im(d)\cap\W_{\R}^{2,2}X\times \Gamma( {\rm End}_0 \, E)=\Y.
\end{align}
where $\Gamma({\rm End}_0 \, E) = \{ v \in \Gamma({\rm End} \,E) : v^\dagger = v, \ \ {\rm Tr} \, v = 0 \}$. This again is needed since shifts $v \mapsto v+C \, {\rm id}_E$ by constant multiples of the identity are not seen by this ansatz. 
    
This ansatz is automatically conformally balanced, the map $\F_{\mathrm{BC}}$ is then constructed by including the remaining equations in the system.
\begin{align}
    &\F_{\mathrm{BC}}(\textrm{X},\textrm{Y})=\mqty[i\partial\bar{\partial}\Tilde{\w}-\alpha'(\Tr \, \Tilde{F}_{\tilde{h}}\wedge \Tilde{F}_{\tilde{h}}-\Tr \, \Tilde{R}_{\Tilde{\w}}\wedge \Tilde{R}_{\Tilde{\w}})\\|| \W||_{\Tilde{\w}} e^{v/2} \, \Tilde{\w}^2\wedge i\Tilde{F}_{\tilde{h}} \, e^{-v/2}].
\end{align}
It was shown in \cite{PW24} that the linearization of $\F_{\mathrm{BC}}$ is surjective at the K\"ahler point, in particular, the linearization 
\begin{align}
    \eval{D_{\textrm{Y}}\F_{\mathrm{BC}}}_{(0,0)}&: C^{k+2,\gamma}(\partial\bar{\partial}\W_{\R}^{1,1}) \times C^{k+2,\gamma}({\rm End}_0 \, E)  \longrightarrow C^{k,\gamma}(\partial\bar{\partial}\W_{\R}^{1,1}) \times C^{k,\gamma}(V(E)) \nonumber\\
    &\eval{D_{\textrm{Y}}\F_{\mathrm{BC}}}_{(0,0)}(\dot{\Theta},\dot{v})=\mqty[L^{\mathrm{BC}}_1 & 0\\ C & L_2^{\mathrm{BC}}]\mqty[\dot{\Theta}\\\dot{v}]
\end{align}
is bijective. Here $L^{\mathrm{BC}}_1$ is the linearization of the conformally balanced condition, which takes the form (suppressing H\"older norms)
\begin{align}
    L^{\mathrm{BC}}_{1}&:\im(d)\cap \W^{2,2}\to \im(d)\cap \W^{2,2}\nonumber \\
    L^{\mathrm{BC}}_{1}\bullet&=-\frac{1}{2||\W||_{\w}}\Delta\bullet.
\end{align}
which is invertible by usual elliptic theory.
\begin{remark}
    In the derivation of $L_{1}^{\mathrm{BC}}$ in \cite{PW24}, we need the $sl_2$ commutation relation $[L_{\w},\Lambda_{\w}]=H$, where $H=\sum_{k=0}^{2n}(k-n)\Pi^{k}$ acting on $\W^kX$. Later we will use this to derive Lemma \ref{lemma-C2} for general $n$-fold cases.
\end{remark}

For the $L_{2}^{\mathrm{BC}}$, it follows the calculation in \cite{PW24} that (suppressing H\"older norms)
\begin{align}
    L_{2}&:\Gamma(\End_0 E)\to V(E)\nonumber \\
    L_{2}\bullet&=\bbi\Lambda_{\w}\bar{\D}\D\bullet\otimes ||\W||_{\w}\frac{\w^3}{3!}
\end{align}
which is shown to be invertible. 

Therefore, by the implicit function theorem, we arrive at the following theorem about the Bott-Chern parameter of the heterotic moduli.
\begin{theorem}
    Near a K\"ahler point $(\w,h)$ on the heterotic moduli with fixed complex structure, for sufficiently small $\alpha'$, there exists a smooth family of solutions to the Hull-Strominger system parameterized by the Bott-Chern cohomology class $\mathfrak{b}\in H^{2,2}_{\mathrm{BC}}(X,\R)$ and the holomorphic class $\alpha\in H^{1}(X,\End E)$.
\end{theorem}

\subsubsection{Aeppli-Bott-Chern duality} By the Poincar\'e-type duality \eqref{eqn:Poincare-Duality} between the Aeppli cohomology and Bott-Chern cohomology, 
\begin{align}
    H^{1,1}_{\mathrm{A}}(X,\C)\times H^{2,2}_{\mathrm{BC}}(X,\C)&\to \C,\nonumber\\
    (\mathfrak{a},\mathfrak{b})&\mapsto\int_{X}\mathfrak{a}\wedge\mathfrak{b},
\end{align}
the dimension of the Aeppli parameter space matches the dimension of the Bott-Chern parameter space.

This matches our expectation in \cite{PW24} that both Aeppli and Bott-Chern parameters should give equivalent descriptions of the heterotic moduli at a K\"ahler point, i.e. the dimension of the Aeppli parameter space $H^{1,1}_{\mathrm{A}}(X)$ matches the dimension of the Bott-Chern parameter space $H^{2,2}_{\mathrm{BC}}(X)$. In particular, our construction of the Aeppli parameter in this paper does not require the introduction of the auxiliary gauge connection on the tangent bundle. Hence, we see the expected dimension match of the local heterotic moduli near a K\"ahler point in two parametrizations.

\section{Open Questions}\label{sec:5}
In the paper, we have shown the existence of the Aeppli parameter without the auxiliary gauge connection on the tangent bundle, and we have also given a brief parallel comparison between the Aeppli parameter and the Bott-Chern parameter. We end the paper with some immediate open questions and future directions.
\begin{enumerate}
    \item Uniqueness is not discussed in the scope this paper. It would be an open question even in this restricted scenario. One might hope that non-linear PDE analysis could shred some lights on this. Positive conclusion of uniqueness would result the extension of Yau's theorem and DUY theorem to this non-K\"ahler case. 
    \item Since we have shown that the Aeppli parameter and Bott-Chern parameter give equivalent local description on the heterotic moduli, one should find a coordinate transformation between them and could calculate the explicit local metric. It would be interesting to compute the metric in both parameterizations and compare with the results from previous works \cites{CdlOM17,G-FRT20}.
    \item We have constructed the Aeppli parameter dual to the Bott-Chern parameter and also briefly discussed the extension of the Aeppli parameter to $n$-fold Hull-Strominger system. However, it is not straightforward to extend the Bott-Chern parameter construction in \cite{PW24}. By the duality, we conjecture that such a generalization should be possible. How to resolve this imbalance in $n$-fold cases discussed in Appendix \ref{App:C}?
    \item We have ignored the variation of complex structure for simplcity. But how to incorporate the deformation of the complex structure into this argument to conclude the existence and local smoothness of the full heterotic moduli space? 
    \item Arriving at Theorem \ref{thm:joint} requires the assumption that the deformation of the stable bundle $E$ is unobstructed. Then how to probe the neighborhood of a K\"ahler point when the deformation of $E$ is obstructed? This would mean that the moduli space of the stable bundle $E$ is singular, would the heterotic moduli space also be singular? Then how to understand the Aeppli parameter and Bott-Chern parameter in this case?
    \item How to probe away from the K\"ahler point, i.e. when the reference geometry is non-K\"ahler? One need to use more refined non-linear PDE analysis to understand the local structure of the heterotic moduli near a non-K\"ahler point.
    \item Physical arguments \cites{CdlOM17,CdlOM19} suggest the K\"ahler structure on the heterotic moduli. It would be interesting to understand the precise condition mathematically. Furthermore, a recent result in \cite{MY26} gives explicit form for the heterotic moduli space up to order ${\alpha'}^2$ with the K\"ahler potential protected and unchanged. Is the K\"ahler potential protected up to higher order or even in full order? It would be interesting to understand the protected K\"ahler potential non-perturbatively. These questions involve with the global theory of the heterotic moduli, which is unclear and much harder than the local theory.
\end{enumerate}

\appendix
\section{The Second Bott-Chern Characteristics and Smoothness of $\F$}\label{App:B}
Regarding the canonical construction of the Bott-Chern secondary characteristics, we can define it via a $4$th-order self-adjoint elliptic operator called the Kodaira-Spencer operator $\ks$
\begin{align}
    \ks&:\W^{2,2}(X)\to \W^{2,2}(X), \nonumber\\
    \ks&=\p\bp\bpd\pd+\bpd\pd\p\bp+\bpd\p\pd\bp+\pd\bp\bpd\p+\bpd\bp+\pd\p,
\end{align}
where the adjoint $\dagger$ is with respect to $\w$. By the elliptic theory, there exists a unique solution, say $\gamma\in(\ker{\ks})^{\perp}$ solving
\begin{align}
    \ks(\gamma)=\Tr \Tilde{F}_{\Tilde{h}}^{2}-\Tr F_{h}^2,
\end{align}
provided that the right-hand side is perpendicular to $\ker\ks$. A well-known integrating by part argument leads
\begin{align}
    \ker\ks=\{\varphi\in\W^{2,2}X:d\varphi=0,\bpd\pd\varphi=0\}.
\end{align}
Then we may solve for $\gamma$ for some $\rho$ and $\zeta$ that
\begin{align}
    \ks(\gamma)=\bbi\p\bp\rho=d\zeta=\Tr \Tilde{F}_{\Tilde{h}}^{2}-\Tr F_{h}^2,
\end{align}
where the existence of $\zeta$ is by Chern-Weil theory in the last equality and the existence of $\rho$ is by $\p\bp$-lemma in second last equality. Then iterative integration by parts shows this equation implies $d\gamma=0$. Therefore,
\begin{align}
    \p\bp\bpd\pd\gamma=\Tr \Tilde{F}_{\Tilde{h}}^{2}-\Tr F_{h}^2,
\end{align}
we can take
\begin{align}
    C_{2}[\tilde{h},h]:=-\bbi\bpd\pd\gamma.
\end{align}
In other words, canonically,
\begin{align}\label{eqn:C2a}
    C_{2}[\tilde{h},h]=-\bbi\bpd\pd \ks^{-1}\qty(\Tr \Tilde{F}_{\Tilde{h}}^{2}-\Tr F_{h}^2).
\end{align}
Similarly, 
\begin{align}\label{eqn:C2b}
    C_{2}[\tw,\w]=-\bbi\bpd\pd \ks^{-1}\qty(\Tr \Tilde{R}_{\tw}^{2}-\Tr R_{\w}^2)
\end{align}

We also give a brief argument (which can be seen also in \cite{PW24}) to show that $\F$ is a smooth map. In the construction of $\F$, the only non-trivial part is the differences between the second Bott-Chern characteristics in the differentiation $C_{2}[\tilde{A}_{\tilde{h}}+t\alpha,A_h]$ where $\alpha\in \Gamma_{\R}(\End E)$ and $\gamma\in \W^{1,1}_{\R}X$. We will show that they are smooth in $s$ and $t$ respectively. We only give the argument for the first one, as the second one follows similarly. First, given $\tilde{A},A,\alpha\in C^{k+1,\gamma}$, since $\ks^{-1}:C^{k,\gamma}\to C^{k+4,\gamma}$, then by the construction \eqref{eqn:C2a}, we have $C_2[\tilde{A}_{\tilde{h}}+s\alpha,A_h]\in C^{k+2,\gamma}$. Define the linearization of $C_{2}$ at $t=0$ by,
\begin{align}
    L_{\tilde{A}}\alpha=\left. \frac{d}{dt}\right|_{t=0}\Tr\qty(F_{\tilde{A}_{\tilde{h}}+t\alpha}\wedge F_{\tilde{A}_{\tilde{h}}+t\alpha}).
\end{align}
Then we desire the show the vanishing of the differential quotient in the limit of $\alpha\to 0$:
\begin{align}
    \lim_{\alpha\to 0}\frac{||C_{2}[\tilde{A}_{\tilde{h}}+\alpha,A_{h}]-C_{2}[\tilde{A}_{\tilde{h}},A_{h}]-\bpd\pd \ks^{-1}\qty(L_{\tilde{A}}\alpha)||_{C^{k+2,\gamma}}}{||\alpha||_{C^{k+1,\gamma}}}=0.
\end{align}
This is done via \eqref{eqn:C2a}, we have the numerator is given by
\begin{align}
    ||\cdots||_{C^{k+2,\gamma}}=\left|\left|\bpd\pd \ks^{-1}\qty(\Tr\qty(F_{\tilde{A}_{\tilde{h}}+\alpha}\wedge F_{\tilde{A}_{\tilde{h}}+\alpha})-\Tr\qty(F_{\tilde{A}_{\tilde{h}}}\wedge F_{\tilde{A}_{\tilde{h}}})-L_{\tilde{A}}\alpha)\right|\right|_{C^{k+2,\gamma}}
\end{align}
By the standard elliptic estimate for the 4-th order elliptic operator $\ks$, we have
\begin{align}
    ||\xi||_{C^{k+4,\gamma}}\leq C||\ks (\xi)||_{C^{k,\gamma}},\quad \forall \xi\in\ker(\ks)^{\perp}.
\end{align}
Hence, we have a bound on the inverse $\ks^{-1}$ that
\begin{align}
    ||\ks^{-1}(\eta)||_{C^{k+4,\gamma}}\leq C||\eta||_{C^{k,\gamma}},\quad \forall \eta: \xi=\ks^{-1}(\eta)\in \ker(\ks)^{\perp}.
\end{align}
The numerator is then bounded by
\begin{align}
    ||\cdots||_{C^{k+2,\gamma}}&\leq C\left|\left|
        \Tr\qty(F_{\tilde{A}_{\tilde{h}}+\alpha}\wedge F_{\tilde{A}_{\tilde{h}}+\alpha})-\Tr\qty(F_{\tilde{A}_{\tilde{h}}}\wedge F_{\tilde{A}_{\tilde{h}}})-L_{\tilde{A}}\alpha
    \right|\right|_{C^{k,\gamma}},\nonumber\\
    &\leq C\left|\left|
        \int^{1}_{0}\frac{d}{dt}\Tr\qty(F_{\tilde{A}_{\tilde{h}}+t\alpha}\wedge F_{\tilde{A}_{\tilde{h}}+t\alpha})dt - \left.\frac{d}{dt}\right|_{t=0}\Tr\qty(F_{\tilde{A}_{\tilde{h}}+t\alpha}\wedge F_{\tilde{A}_{\tilde{h}}+t\alpha})
    \right|\right|_{C^{k,\gamma}},\nonumber\\
    &\leq C\left|\left|
        \int^{1}_{0}\qty(\Tr F_{\tilde{A}_{\tilde{h}}+t\alpha}d_{\tilde{A}+t\alpha}\alpha-\Tr F_{\tilde{A}_{\tilde{h}}}d_{\tilde{A}}\alpha) dt
    \right|\right|_{C^{k,\gamma}},\nonumber\\
    &= C\left|\left|
        \int^{1}_{0}\int^{1}_{0}\frac{d}{ds}\qty(\Tr F_{\tilde{A}_{\tilde{h}}+st\alpha}d_{\tilde{A}_{\tilde{h}}+st\alpha}\alpha) ds dt
    \right|\right|_{C^{k,\gamma}}.
\end{align}
Then it follows here that the numerator is bounded by $C||\alpha||^2_{C^{k+1,\gamma}}$. Hence the limit follows, and the smoothness of $\F$ is established as the term concerning $C_{2}[\tw,\w]$ follows similarly.

\section{Reflection Formula}\label{App:A}
In this appendix, we record and apply the reflection formula from Huybrechts \cite{H15}. In the local theory, we have the primitive decomposition:
\begin{proposition}
    $(V,\langle \bullet,\bullet \rangle,J)$ Euclidean vector space of $\dim_{\R}(V)=n$. Then
    \begin{align}
        \W^{k}V^{*}=\bigoplus_{i\geq 0}L^{i}(P^{k-2i}).
    \end{align}
\end{proposition}
Then for a primitive $k$-form $\alpha\in P^{k}$, we have ``a mysterious but extremely useful'' reflection formula.
\begin{proposition}[Reflection formula]
    For each $\alpha\in P^{k}$, 
    \begin{align}
        \star L^{j}\alpha=(-1)^{\frac{1}{2}k(k+1)}\frac{j!}{(n-k-j)!}L^{n-k-j}J(\alpha),
    \end{align}
    where $J:\W^{*}V_{\C}\to \W^{*}V_{\C}$, the complex structure act on cotangent bundle as $J=\sum_{p,q}\bbi^{p-q}\Pi^{p,q}$, and $\Pi^{p,q}:\W^{*}V_{\C}\to \W^{p,q}V$ the projection to type $(p,q)$.
\end{proposition}
Note, a direct computation show that on $\W^{k}X$,
\begin{align}
    J^{2}&=\sum_{p,q}\bbi^{2(p-q)}\Pi^{p,q}
    =\sum_{p,q}(-1)^{p-q}\Pi^{p,q}
    =\begin{cases}
        +\id & $k=p+q$$\text{ even}$,\\
        -\id & $k=p+q$$\text{ odd}$.\\
    \end{cases}
\end{align}
The primitive decomposition extends directly to global theory.
\begin{corollary}
    Let $(X,g)$ be a hermitian manifold, there is a direct sum decomposition,
    \begin{align}
        \W^{k}X=\bigoplus_{i\geq 0}L^{i}(P^{k-2i}X),
    \end{align}
    where $P^{k-2i}X=\ker(\Lambda:\W^{k-2i}X\to \W^{k-2i-2}X)$. This is also compatible with bidegree decomposition $\W^{k}_{\C}X=\bigoplus_{p+q=k}\W^{p,q}X$,
    \begin{align}
        P^{k}_{\C}X=\bigoplus_{p+q=k}P^{p,q}X,
    \end{align}
    where $P^{p,q}X:=P^{p+q}_{\C}X\cap \W^{p,q}X$.
\end{corollary}
The reflection formula extends to global forms as well.

We record the following specialization of the reflection formula in 1- and 2-forms in the general $n$-fold case.
\begin{lemma}[$1$-form reflection formula in $n$-fold case]
    For each $\gamma\in \W^1X$, we have
    \begin{align}
        \star \gamma = -\frac{1}{(n-1)!}L^{n-1}J\gamma.
    \end{align}
    We can also rewrite it as
    \begin{align}\label{eqn:1-form-reflection-n-fold}
        L^{n-1}\gamma = (n-1)!\star J\gamma.
    \end{align}
\end{lemma}

\begin{proof}
    Since $\Lambda_{\w}\gamma$ for any $1$-form $\gamma$, hence all $1$-forms are primitive. Then the proof follows the application of the general reflection formula for $1$-forms with $\alpha=\gamma$, $n=n$, $k=1$, $j=0$.
    \begin{align}
        \star \gamma &= (-1)^{\frac{1}{2}\cdot 1\cdot 2}\frac{0!}{(n-1)!}L^{n-1}J(\gamma) =-\frac{1}{(n-1)!}L^{n-1}J(\gamma)
    \end{align}
    Then the rewritten form follows and $J^2=-\id$ on $\W^1X$,
    \begin{align}
        \star J\gamma & = -\frac{1}{(n-1)!}L^{n-1}J^2\gamma = \frac{1}{(n-1)!}L^{n-1}\gamma.
    \end{align}
\end{proof}

\begin{lemma}[$2$-form reflection formula in $n$-fold case]
    For any 2-form $\gamma\in \W^{2}X$, 
    \begin{align}
        \star \gamma = \frac{1}{(n-1)!}(\Lambda_{\w}\gamma)L^{n-1}-\frac{1}{(n-2)!}L^{n-2}J\gamma.
    \end{align}
    We can also rewrite it as
    \begin{align}
        L^{n-2}\gamma = -(n-2)!\star J\gamma + \frac{1}{(n-1)}(\Lambda_{\w}J\gamma)L^{n-1}. \label{eqn:2-form-reflection-n-fold}
    \end{align}
\end{lemma}

\begin{proof}
    Follow the same argument as before. Any $2$-form $\gamma\in \W^2X$ admits a primitive decomposition,
    \begin{align}
        \gamma = \frac{1}{n}(\Lambda_{\w}\gamma)\w + \beta,\quad \beta\in P^2X,
    \end{align}
    where the coefficient $\frac{1}{n}$ arises from the fact that $\Lambda_{\w}\w=n$ in $n$-fold $X$. Then applying the reflection formula to $\w$ and $\beta$ respectively, we have
    \begin{align}
        \star \w &= \frac{1}{(n-1)!}\w^{n-1}\quad (\alpha=1,n=n,k=0,j=1),\\
        \star \beta &= -\frac{1}{(n-2)!}L^{n-2}J\beta\quad (\alpha=\beta,n=n,k=2,j=0).
    \end{align}
    Then collecting the result, we have,
    \begin{align}
        \star \gamma & = \star \qty( \frac{1}{n}(\Lambda_{\w}\gamma)\w +\beta  ),\nonumber\\
        &=\frac{1}{n!}(\Lambda_{\w}\gamma)L_{\w}^{n-1}-\frac{1}{(n-2)!}L_{\w}^{n-2}J\beta,\nonumber\\
        &=\frac{1}{n!}(\Lambda_{\w}\gamma)L_{\w}^{n-1}-\frac{1}{(n-2)!}L_{\w}^{n-2}J\qty( \gamma-\frac{1}{n}(\Lambda_{\w}\gamma)\w ),\nonumber\\
        &=\frac{1}{(n-1)!}(\Lambda_{\w}\gamma)L^{n-1}-\frac{1}{(n-2)!}L^{n-2}J\gamma.
    \end{align}
    Then the rewritten form follows and with $J^2=\id$ on $\W^2X$,
    \begin{align}
        \star J\gamma & = \frac{1}{(n-1)!}(\Lambda_{\w}J\gamma)L^{n-1}-\frac{1}{(n-2)!}L^{n-2}\gamma,
    \end{align}
    rearrange it to obtain the result.
\end{proof}

\section{$n$-Fold case}\label{App:C}
In this appendix, we will briefly remark on the $n$-fold case of the Hull-Strominger system. In particular, we will show that the main theorem is special in $n=2,3$-fold case and our construction of both Bott-Chern and Aeppli parameter to general $n$-fold case is obstructed.

To introduce $n$-fold Hull-Strominger system, let $X$ be a compact complex $n$-fold with a nowhere vanishing holomorphic $(n,0)$-form $\W$, and a hermitian metric $\w$. Let $E\to X$ be a holomorphic vector bundle with the associated hermitian metric $h$, the associated Chern connection $A_{h}$ and the Chern curvature $F_{h}$. The Hull-Strominger system on an $n$-fold reads as
\begin{align}
    d(||\W||_{\w}\w^{n-1})&=0\quad\text{Conformally balanced condition},\\
    F_{h}\wedge \w^{n-1}=0,\ F^{0,2}_{h}&=0\quad\text{Hermitian-Yang-Mills conditions},\\
    \bbi\p\bp\w-\alpha'(\Tr R_{\w}^2 - \Tr F_{h}^{2})&=0\quad \text{Anomaly cancellation condition}. 
\end{align}

\subsection{Aeppli parameter}
The construction of $\F$ is similar in $3$-fold case. 
\begin{align}
    \F(\mathrm{X},\mathrm{Y})=\mqty[
        \tilde{\w}_{u}-\chi\\
        \star d(||\W||_{\chi}\chi^{n-1})\\
        ||\W||_{\chi}e^{\frac{v}{2}}(\chi^{n-1}\wedge \bbi\tilde{F}_{v})e^{-\frac{v}{2}}-C\chi^{n}\otimes \id
    ],
\end{align}
with the normalization constant
\begin{align}
    C=\frac{1}{\mathrm{Rk}(E)}\frac{\int_{X}||\W||_{\chi}\chi^{n-1}\wedge \bbi\tilde{F}_{\tilde{h}}}{\int_{X}||\W||_{\chi}\chi^n},
\end{align}
where the deformations are given by
\begin{align}
    \tilde{\w}_{u}&=\w+u,\\
    \chi&=\w+\mathfrak{a}+\theta(\xi)-\alpha'(C_{2}[\tilde{h},h]-C_{2}[\tw_{u},\w]+\beta),\quad \theta(\xi)=(1+J)d\xi,\\
    \tilde{h}&=he^{v}.
\end{align}
The corresponding Banach spaces are given by
\begin{align}
    \mathcal{X}&=\R\times H^{1,1}_{\mathrm{A}}(X),\\
    \mathcal{Y}&=C^{k,\gamma}(\W^{1,1}_{\R}X)\times C^{k+2,\gamma}(\im(d^{\dagger})\cap \W^{1}_{\R}X)\times C^{k+2,\gamma}(\Gamma_{\R}(\End_0 E)),\\
    \mathcal{Z}&=C^{k,\gamma}(\W^{1,1}_{\R}X)\times C^{k,\gamma}(\im(d^{\dagger})\cap\W^{1}_{\R}X)\times C^{k,\gamma}(V(E)).
\end{align}
Again, the surjectivity of the linearization $D\F$ is given by the bijectivity of the block operator $D_{\mathrm{Y}}\F$ evaluated at the K\"ahler point,
\begin{align}
    \left. D_{\mathrm{Y}}\F \right|_{(\mathrm{X},\mathrm{Y})=(0,0)}\mqty[\dot{u} \\ \dot{\xi} \\ \dot{v}] = \mqty[
        \id_{\W^{1,1}_{\R}X} & L_{4} & 0\\
        0 & L_1 & 0\\
        0 & L_3 & L_2
    ].
\end{align}
This means we need to show the invertibility of $L_1$ and $L_2$ as in the $3$-fold case. The operator $L_1$ can be calculated and shown to be invertible in similar fashion.

Since we follow the similar construction and deformation ansatz, we will need to calculate the linearization of the $\star$-dualized conformally balanced operator $L_1$. We will now see that there is a remainder term as the obstruction of the ellipticity and hence the invertibility of $L_1$ in general $n$-fold case.
\begin{lemma}\label{lemma-C1}
    The linearization of the balanced metric deformation $L_{1}$ is given by
    \begin{align}
        L_{1}\dot{\xi}=(n-1)!||\W||_{\w}\Delta\dot{\xi}.
    \end{align}
    Moreover,
    \begin{align}
        L_{1}\bullet=(n-1)!||\W||_{\w}\Delta\bullet: \im(d^{\dagger})\cap \W^{1}_{\R}X\to \im(d^{\dagger})\cap \W^{1}_{\R}X,
    \end{align}
    is invertible.
\end{lemma}

\begin{proof}
    We need to recalculate $\delta_{0}||\W||_{\chi}$ in $n$-fold case. Take the identity
    \begin{align}
        ||\W||_{\chi}^{2}\frac{\chi^n}{n!}=(-1)^{n}(\bbi)^{n}\W\bar{\W},
    \end{align}
    and take the variation at both sides,
    \begin{align}
        0=\delta_{0}\qty(||\W||_{\chi}^{2}\frac{\chi^n}{n!})=2||\W||_{\w}\delta_{0}(||\W||_{\chi})\frac{\w^n}{n!}+\frac{1}{(n-1)!}||\W||_{\w}^2\w^{n-1}\wedge \delta_{0}\chi.
    \end{align}
    This implies that
    \begin{align}
        \delta_{0}||\W||_{\chi}=-\frac{\frac{1}{(n-1)!}||\W||_{\w}^2\w^{n-1}\wedge \dot{\chi}}{\frac{2}{n!}||\W||_{\w}\w^n}=-\frac{n}{2}\frac{||\W||_{\w}\w^{n-1}\wedge \dot{\chi}}{\w^n}.
    \end{align}
    Using the identity that for any $2$-form $\alpha\in \W^{2}X$, 
    \begin{align}
        (\Lambda_{\w}\alpha)\w^n=n\alpha\wedge \w^{n-1}\implies \w^{n-1}\wedge\alpha=\frac{1}{n}(\Lambda_{\w}\alpha)\w^n.
    \end{align}
    We hence obtain that
    \begin{align}
        \delta_{0}||\W||_{\chi}=-\frac{1}{2}||\W||_{\w}(\Lambda_{\w}\dot{\chi}).
    \end{align}
    Therefore, we have
    \begin{align}
        L_{1}\dot{\xi}=\star d\qty(-\frac{1}{2}||\W||_{\w}(\Lambda_{\w}\dot{\chi})\w^{n-1})+(n-1)\star d(||\W||_{\w}\w^{n-2}\wedge \dot{\chi}).
    \end{align}
    Normalize without the constant $||\W||_{\w}$ factor, 
    \begin{align}
        ||\W||_{\w}^{-1}L_{1}\dot{\xi}=\star d\qty(-\frac{1}{2}(\Lambda_{\w}\dot{\chi})\w^{n-1})+(n-1)\star d(\w^{n-2}\wedge \dot{\chi}).
    \end{align}
    We can rearrange the equation into the following form, since $[L_{\w},d]=0$,
    \begin{align}
        ||\W||_{\w}^{-1}L_{1}\dot{\xi}=(n-1)\star\qty(
            \w^{n-2}\wedge d\qty( \dot{\chi}-\frac{1}{2(n-1)}\w(\Lambda_{\w}\dot{\chi}) )
        )=(n-1)\star L_{\w}^{n-2}dT(\dot{\chi}),
    \end{align}
    where now the $T$-operator is defined as,
    \begin{align}
        T(\bullet):=\id\bullet-\frac{1}{2(n-1)}L_{\w}\Lambda_{\w}\bullet.
    \end{align}
    Expand the operator $T$ and expression of $\chi$, we have
    \begin{align}
        \frac{1}{(n-1)}||\W||_{\w}^{-1}L_{1}&=\star L_{\w}^{n-2}dT(1+J)d\dot{\xi},\nonumber\\
        &=\star L_{\w}^{n-2}d(1+J)d\dot{\xi}-\frac{1}{2(n-1)}\star L_{\w}^{n-2}dL_{\w}\Lambda_{\w}(1+J)d\dot{\xi},\nonumber\\
        &=\star L_{\w}^{n-2}d(1+J)d\dot{\xi}-\frac{1}{(n-1)}\star L_{\w}^{n-2}dL_{\w}\Lambda_{\w}\qty(d\dot{\xi})^{1,1},
    \end{align}
    where in the last line, we recall that 
    \begin{align}
        (1+J)d\dot{\xi}=2(d\dot{\xi})^{1,1}.
    \end{align}
    Then we can further rearrange the equation into the following form, since $[L_{\w},d]=0$ and $\Lambda_{\w}(d\dot{\xi})^{1,1}=\Lambda_{\w}d\dot{\xi}$,
    \begin{align}
        \frac{1}{(n-1)}||\W||_{\w}^{-1}L_{1}&=\star L_{\w}^{n-2}dJd\dot{\xi}-\frac{1}{(n-1)}\star L_{\w}^{n-2}dL_{\w}\Lambda_{\w}d\dot{\xi}=\star L_{\w}^{n-2}dJd\dot{\xi}-\frac{1}{(n-1)}\star L_{\w}^{n-1}d\Lambda_{\w}d\dot{\xi}.
    \end{align}
    Again, for the first term, we use the identity that,
    \begin{align}
        J\circ d=-d_{c}\circ J,\quad d_{c}=-\bbi(\p-\bp),
    \end{align}
    and for the second term, we use $d^2=0$ and the K\"ahler identity that,
    \begin{align}
        [\Lambda_{\w},d]=-d_{c}^{\dagger}.
    \end{align}
    We arrive at the following expression,
    \begin{align}
        \frac{1}{(n-1)}||\W||_{\w}^{-1}L_{1}&=-\star L_{\w}^{n-2}dd_{c}J\dot{\xi}+\frac{1}{(n-1)}\star L_{\w}^{n-1}dd_{c}^{\dagger}\dot{\xi}.
    \end{align}
    Switch the position of $d$ and $L_{\w}^{n-2}$ in the first term, we have
    \begin{align}
        \frac{1}{(n-1)}||\W||_{\w}^{-1}L_1&=-\star d L_{\w}^{n-2} d_{c} J\dot{\xi} + \frac{1}{(n-1)}\star L_{\w}^{n-1} d d_{c}^{\dagger} \dot{\xi}.
    \end{align}
    Take 2-form $\gamma = d_{c} J \dot{\xi}$, and apply the reflection formula \eqref{eqn:2-form-reflection-n-fold} (See Appendix. \ref{App:A} for derivations)
    \begin{align}
        L^{n-2}\gamma = -(n-2)!\star J\gamma + \frac{1}{(n-1)}(\Lambda_{\w}J\gamma)L^{n-1}.
    \end{align}
    We then have
    \begin{align}
        L_{\w}^{n-2}d_{c} J \dot{\xi}=-(n-2)!\star Jd_{c}J\dot{\xi}+\frac{1}{(n-1)}(\Lambda_{\w}Jd_{c}J\dot{\xi})L_{\w}^{n-1}.
    \end{align}
    By the identity that $J\circ d_{c}= d\circ J$ and $J^2=-\id$ on $1$-forms, we can further simplify,
    \begin{align}
        L_{\w}^{n-2}d_{c} J \dot{\xi}=-(n-2)!\star dJ^2\dot{\xi}+\frac{1}{(n-1)}(\Lambda_{\w}dJ^2\dot{\xi})L_{\w}^{n-1}=(n-2)!\star d\dot{\xi}-\frac{1}{(n-1)}(\Lambda_{\w}d\dot{\xi})L_{\w}^{n-1}.
    \end{align}
    Then we can apply the K\"ahler identity $[\Lambda_{\w},d]=-d_{c}^{\dagger}$ to the first term again and note that $\Lambda_{\w}\dot{\xi}=0$, we have
    \begin{align}
        L_{\w}^{n-2}d_{c} J \dot{\xi}&=(n-2)!\star d\dot{\xi}+\frac{1}{(n-1)}(d_{c}^{\dagger}\dot{\xi})L_{\w}^{n-1}.
    \end{align}
    Hence, 
    \begin{align}
        \frac{1}{(n-1)}||\W||_{\w}^{-1}L_1&=-\star d\qty( (n-2)!\star d\dot{\xi}+\frac{1}{(n-1)}(d_{c}^{\dagger}\dot{\xi})L_{\w}^{n-1} ) + \frac{1}{(n-1)}\star L_{\w}^{n-1} d d_{c}^{\dagger} \dot{\xi},\nonumber\\
        &=-(n-2)!\star d\star d\dot{\xi}=(n-2)!d^{\dagger} d\dot{\xi}=(n-2)!\Delta\dot{\xi}.
    \end{align}
    Hence, rearranging the constant factor, we have
    \begin{align}
        L_1\dot{\xi}=(n-1)!||\W||_{\w}\Delta\dot{\xi}.
    \end{align}
    Then by standard elliptic theory and Hodge decomposition argument, we can conclude that $L_1:\im(d^{\dagger})\cap \W^{1}_{\R}X\to \im(d^{\dagger})\cap \W^{1}_{\R}X$ is bijective and hence invertible.
\end{proof}

For $L_2$, the extension to $n$-fold case follows easily that
\begin{align}
    L_{2}\bullet&=\bbi\Lambda_{\w}\bar{\D}\D\bullet\otimes ||\W||_{\w}\frac{\w^n}{n!},
\end{align}
and the invertibility follow similar as in $3$-fold case in \cite{PW24}.

Combination of the above discussion and the implicit function theorem, we can extend the main results to the $n$-fold Hull-Strominger system. 

\subsection{Bott-Chern parameter}
There is an imbalance in whether one can extend the construction of the moduli parameter to $n$-fold between the Aeppli and Bott-Chern cases. We have shown that the Aeppli parameter can be naturally extended to $n$-fold case. However, this is not obvious for the Bott-Chern parameter. In particular, the linearized operator $L_{1}^{\mathrm{BC}}$ will not be Laplacian unless $n=3$.

To expand on the Bott-Chern parameter briefly, we can follow the calculation in \cite{PW24} for $L_{1}^{\mathrm{BC}}$, and use the following lemma which can be easily seen by induction on $k$.
\begin{lemma}\label{lemma-C2}
    We have
    \begin{align}\label{eqn:Lambda-L-k}
        \Lambda_{\w}L_{\w}^{k}=L_{\w}^{k}\Lambda_{\w}-kL_{\w}^{k-1}(H+k-1).
    \end{align}
\end{lemma}

As a result, we obtain for the general $n$-fold, the variation of balanced metric reads
\begin{align}
    \delta_{0}(\tw^{n-2})&=\frac{1}{(n^2-4n+5)}||\W||_{\w}^{-1}\Lambda_{\w}\delta_0\qty(||\W||_{\tw}\tw^{n-1}),\nonumber\\
    &=\frac{1}{(n^2-4n+5)}||\W||_{\w}^{-1}\Lambda_{\w}\dot{\Theta}.
\end{align}
It is clear to see that when $n=3$, we have
    \begin{align}
        \delta_0\tw&=\frac{1}{2}||\W||_{\w}^{-1}\Lambda_{\w}\delta_{0}\qty(||\W||_{\tw}\tw^{2})
    \end{align}
    Then
    \begin{align}
        L_{1}^{\mathrm{BC}}\dot{\Theta}&=\delta_0\qty(\bbi\p\bp\tw)=\bbi\p\bp\delta_0\tw=\frac{\bbi}{2}\p\bp\qty(||\W||_{\w}^{-1}\Lambda_{\w}\delta_{0}\qty(||\W||_{\tw}\tw^{2})),\nonumber\\
        &=\frac{\bbi}{2||\W||_{\w}}\p\bp\Lambda_{\w}\dot{\Theta}=-\frac{1}{2||\W||_{\w}}\p\pd_{\w}\dot{\Theta}=-\frac{1}{2||\W||_{\w}}\Delta_{\w}\dot{\Theta}
    \end{align}
    where we have used the K\"ahler identity $[\Lambda_{\w},\bp]=-\bbi\pd$.
But, this will not work for the general $n\neq 3$, and we do not have the Laplacian.

Nevertheless, this variation result seems to suggest we should generalize the anomaly cancellation condition to $n$-fold case as
\begin{align}
    \bbi\p\bp\tw^{n-2}-\O(\alpha')=0.
\end{align}
Then the Bott-Chern parameter holds and extends to this generalized $n$-fold Hull-Strominger system, and the linearization operator returns to the Laplcian. If so, on the dual side, the corresponding Aeppli parameter is not clear for the extended anomaly cancellation condition. Moreover, if this would be a valid generalization, the precise form of the $\O(\alpha')$ is unclear and the physical meaning of this generalized anomaly cancellation condition is vague. It is worth noting that the discussion of the anomaly flow in general $n$-fold case in \cite{P24} also arrive such similar generalization, 
\begin{align}
    \p_{t}\qty(||\W||_{\w(t)}\w(t)^{n-2})=\bbi\p\bp\w(t)^{n-2}.
\end{align}
It will be interesting to investigate the analytic property of such a generalized anomaly flow.

For our purpose, this begs the question whether the Bott-Chern parameter can be generalized to $n$-fold case or not. Is there any natural generalization of the Hull-Strominger system to general $n$-fold case? Since we have already shown that the Aeppli parameter generalizes natural, then by the duality argument, we expect that such a generalization should exist, however the explicit construction is not clear yet. We will leave it as an open question for future work.


\begin{thebibliography}{99}

\bibitem{AGS14} L. B. Anderson, J. Gray, and E. Sharpe, {\em Algebroids, heterotic moduli spaces and the Strominger system}, Journal of High Energy Physics 2014, no. 7 (2014), 1-40.

\bibitem{AG12}B. Andreas, and M. G-. Fernandez. {\em Solutions of the Strominger system via stable bundles on Calabi-Yau threefolds}. Communications in Mathematical Physics 315 (2012): 153-168.

\bibitem{AdlOMS-CS18} A. Ashmore, X. de la Ossa, R. Minasian, C. Strickland-Constable, and E.E. Svanes, {\em Finite deformations from a heterotic superpotential: holomorphic Chern-Simons and an $L_{\infty}$ algebra}, Journal of High Energy Physics 2018, no. 10 (2018), 1-60.

\bibitem{CdlOM17} P. Candelas, X. de la Ossa, and J McOrist, {\em A metric for heterotic moduli}, Communications in Mathematical Physics 356, no. 2 (2017), 567-612.

\bibitem{CdlOM19} P. Candelas, X. de La Ossa, J. McOrist, and R. Sisca. {\em The universal geometry of heterotic vacua}, Journal of High Energy Physics 2019, no. 2 (2019), 1-47.

\bibitem{CHSW85} P. Candelas, G. Horowitz, A. Strominger, and E. Witten, {\em Vacuum configurations for superstrings}, Nuclear Phys. B 258 (1985), no. 1, 46–74.

\bibitem{CPY22}T. C. Collins, S. Picard, and S. -T. Yau, {\em The Strominger system in the square of a K\"ahler class}, arXiv preprint arXiv:2211.03784 (2022).

\bibitem{dlOS14a}X. de la Ossa, and E. E. Svanes, {\em Holomorphic bundles and the moduli space of $\mathcal{N}=1$ supersymmetric heterotic compactifications}. Journal of High Energy Physics, 2014(10), pp.1-55.

\bibitem{D85} S.K. Donaldson, {\em Anti self-dual Yang-Mills connections over complex algebraic surfaces and stable vector bundles}, Proc. London Math. Soc. (3) 50 (1985), no.1, 1-26.

\bibitem{FHP17} T. Fei, Z. Huang, S. Picard, {\em A construction of infinitely many solutions to the Strominger system}, Journal of Differential Geometry 117(1), 23–39.

\bibitem{FY08} J. Fu, and S.-T Yau, {\em The theory of superstring with flux on non-Kahler manifolds and the complex Monge-Ampere equation}, J. Differential Geom. 78 (2008), no. 3, 369–428.

\bibitem{G-FRST22} M. Garcia-Fernandez, R. Rubio, C. Shahbazi, C. Tipler, {\em Canonical metrics on holomorphic Courant algebroids}, Proc. London Math. Soc. 125 (3) (2022) 700-758.

\bibitem{G-FRT17} M. Garcia-Fernandez, R, Rubio, C. Tipler, {\em Infinitesimal moduli for the Strominger system and Killing spinors in generalized geometry}, Mathematische Annalen 369 (2017) 539-595.

\bibitem{G-FRT20} M. Garcia-Fernandez, R. Rubio, C. Tipler, {\em Holomorphic string algebroids}, Trans. Amer. Math. Soc. 373 (2020) 7347-7382.

\bibitem{G-FRT24} M. Garcia-Fernandez, R. Rubio, C. Tipler, {\em Gauge theory for string algebroids}, Journal of Differential Geometry 128 (1) (2024) 77-152.

\bibitem{H86}C. M. Hull. {\em Compactifications of the heterotic superstring}. Physics Letters B 178, no. 4 (1986): 357-364.

\bibitem{H15}D. Huybrechts. Complex geometry: An Introduction. Vol. 78. Berlin: Springer, 2005.

\bibitem{IP13} S. Ivanov and G. Papadopoulos, {\em Vanishing theorems on $(l|k)$-strong Kaehler manifolds with torsion}, Adv. Math. 237 (2013), 147-164.

\bibitem{KS60} K. Kodaira and D.C. Spencer, {\em On Deformations of Complex Analytic Structures, III. Stability Theorems for Complex Structures}, Annals of Mathematics 71(1) (1960), p.43-76.

\bibitem{LY05} J. Li and S.T. Yau, {\em The existence of supersymmetric string theory with torsion}, J. Diff. Geom. {\bf 70} no. 1 (2005), 143-181.

\bibitem{MS12} D. McDuff and D. Salamon, {\em J-holomorphic curves and symplectic topology}, American Mathematical Society Colloquium Publications, Vol. 52, American Mathematical Society, 2012.

\bibitem{MP25} J. McOrist, and S. Picard, 2025 {\em Stringy Corrections to Heterotic SU (3)-Geometry}, arXiv preprint arXiv:2507.02388.

\bibitem{MY26} J. McOrist and Q. Yin {\em Heterotic moduli, the double extension and the ${\alpha'}^2$ metric}, to appear on ArXiV.

\bibitem{M82} M.L. Michelsohn, {\em On the existence of special metrics in complex geometry}, Acta Math. {\bf 149}, (1982) 261--295.

\bibitem{P24} S. Picard, 2024 {\em The Strominger System and Flows by the Ricci Tensor}, preprint arXiv:2402.17770.

\bibitem{PW24} S. Picard, and P.-L. Wu, 2024 {\em Balanced and Aeppli Parameters for the Heterotic Moduli}, preprint arXiv:2401.05331.

\bibitem{S86} A. Strominger, {\em Superstrings with torsion}, Nuclear Phys. B 274 (1986), no. 2, 253–284.

\bibitem{UY86} K. Uhlenbeck and S.T. Yau {\em On the existence of Hermitian-Yang-Mills connections in stable vector bundles}, Comm. Pure Appl. Math. {\bf 39-S} (1986), 257--293; {\bf 42} (1989), 703--707. 

\bibitem{WW87} L. Witten and E. Witten, {\em Large radius expansion of superstring compactifications}, Nuclear Physics B 281, no. 1-2 (1987), 109-126.

\bibitem{Y78} S.-T. Yau, {\em On the Ricci curvature of a compact K\"ahler manifold and the complex Monge-Amp\`ere equation.} I, Comm. Pure Appl. Math. 31 (1978) 339-411.

\end{thebibliography}
\end{document}